\documentclass{article}

\usepackage{amsmath}
\usepackage{amssymb}
\usepackage{latexsym}
\usepackage{theorem}
\usepackage[all]{xypic}

\title{Algebraic $K$-theory Spectra and Factorisations of Analytic Assembly Maps}
\author{Paul D.~Mitchener \\
Georg-August Universit{\"a}t, G{\"o}ttingen \\
e-mail: mitch@uni-math.gwdg.de}

\newenvironment{proof}{\par \noindent{\bf Proof: }}{\hspace{\stretch{1}} $\Box$ \par \mbox{}}
\newcommand{\noproof}{\hspace{\stretch{1}} $\Box$}
\newtheorem{theorem}{Theorem}[section]
\newtheorem{proposition}[theorem]{Proposition}
\newtheorem{lemma}[theorem]{Lemma}
\newtheorem{corollary}[theorem]{Corollary}
{\theorembodyfont{\rmfamily}
\newtheorem{definition}[theorem]{Definition}

}
\newenvironment{theorem*}{\par \medskip \noindent{\bf Theorem }}{\par \mbox{}}
\newenvironment{lemma*}{\par \medskip \noindent{\bf Theorem }}{\par \mbox{}}

\newcommand{\Hom}{\mathop{Hom}}

\newcommand{\Ob}{\mathop{Ob}}
\newcommand{\Sk}{\mathop{Sk}}

\newcommand{\im}{\mathop{im}}
\newcommand{\Or}{\mathop{Or}}
\newcommand{\map}{\mathop{map}}

\begin{document}

\maketitle

\section*{Abstract}

In this article we use existing machinery to define connective $K$-theory spectra associated to topological ringoids.  Algebraic $K$-theory of discrete ringoids, and the analytic $K$-theory of Banach categories are obtained as special cases.

As an application, we show how the analytic assembly maps featuring in the Novikov and Baum-Connes conjectures can be factorised into composites of assembly maps resembling those appearing in algebraic $K$-theory and maps coming from completions of certain topological ringoids into Banach categories.  These factorisations are proved by using existing characterisations of assembly maps along with our unified picture of algebraic and analytic $K$-theory.

\tableofcontents

\section{Ringoids and Additive Categories}

Recall (see for example \cite{Mi} that a category $\mathcal R$ is called a {\em ringoid} if every morphism set $\Hom (a,b)_{\mathcal R}$ is an abelian group, and composition of morphisms
$$\Hom (b,c)_{\mathcal R}\times \Hom (a,b)_{\mathcal R} \rightarrow \Hom (a,c)_{\mathcal R}$$
is bilinear with respect to group addition.  If $\mathcal R$ is a topological category, where the set of objects carries the discrete topology and the group addition is continuous, we call the category $\mathcal R$ a {\em topological ringoid}.

A functor $F\colon {\mathcal R}\rightarrow {\mathcal S}$ between topological ringoids is termed a {\em homomorphism} of topological ringoids if it is continuous, and a group homomorphism on each morphism set.

Let $\mathcal R$ be a ringoid, and consider objects $a,b\in \Ob ({\mathcal R})$.  Then an object $a\oplus b$ is called a {\em biproduct} of the objects $a$ and $b$ if it comes equipped with morphisms $i_a \colon a\rightarrow a\oplus b$, $i_b \colon b\rightarrow a\oplus b$, $p_a \colon A\oplus B\rightarrow A$, and $p_B \colon a\oplus b \rightarrow b$ satisfying the equations
$$p_a i_a = 1_a \qquad p_b i_b = 1_b \qquad i_a p_a + i_b p_b = 1_{a\oplus b}$$

Observe that a biproduct is simultaneously a product and a coproduct.  A ringoid $\mathcal R$ is called an {\em additive category} if it has a zero object (that is, and object that is simultaneously initial and terminal), and every pair of objects has a biproduct.  A topological ringoid that is an additive category is termed an {\em additive topological category}.

If $\mathcal R$ and $\mathcal S$ are additive topological categories, a topological ringoid homomorphism $F\colon {\mathcal R}\rightarrow {\mathcal S}$ is called an {\em additive functor} if the object $F(a\oplus b)$ is a biproduct of the objects $F(a)$ and $F(b)$ whenever $a,b\in \Ob ({\mathcal R})$.

\begin{definition}
Let ${\mathcal R}$ be a topological ringoid.  Then we define the {\em additive completion}, ${\mathcal R}_\oplus$, to be the category in which the objects are formal sequences of the form
$$a_1 \oplus \cdots \oplus a_n \qquad a_i \in \Ob ({\mathcal R})$$
Repetitions are allowed in such formal sequences.  The empty sequence is also allowed, and labelled $0$.

The morphism set $\Hom (a_1 \oplus \cdots \oplus a_m , b_1 \oplus \cdots \oplus b_n )$ is defined to be the set of matrices of the form
$$\left( \begin{array}{ccc}
x_{1,1} & \cdots & x_{1,m} \\
\vdots & \ddots & \vdots \\
x_{n,1} & \cdots & x_{n,m} \\
\end{array} \right)  \qquad x_{i,j} \in \Hom (a_j , b_i )$$
and composition of morphisms is defined by matrix multiplication.

The set of objects in the category ${\mathcal R}_\oplus$ is given the discrete topology.  The morphism set $\Hom (a_1 \oplus \cdots \oplus a_m , b_1 \oplus \cdots \oplus b_n )$ is topologised as a subspace of the space ${\mathcal R}^{mn}$.
\end{definition}

Given a topological ringoid homomorphism $F\colon {\mathcal R} \rightarrow {\mathcal S}$, there is an induced additive functor $F_\oplus \colon {\mathcal R}_\oplus \rightarrow {\mathcal S}_\oplus$ defined by writing
$$F_\oplus (a_1 \oplus \cdots a_n ) = F(a_1) \oplus \cdots \oplus F(a_n) \qquad a_i \in \Ob ({\mathcal R})$$
and
$$F_\oplus \left( \begin{array}{ccc}
x_{1,1} & \cdots & x_{1,m} \\
\vdots & \ddots & \vdots \\
x_{n,1} & \cdots & x_{n,m} \\
\end{array} \right) =
\left( \begin{array}{ccc}
F(x_{1,1}) & \cdots & F(x_{1,m}) \\
\vdots & \ddots & \vdots \\
F(x_{n,1}) & \cdots & F(x_{n,m}) \\
\end{array} \right)  \qquad x_{i,j} \in \Hom (a_j , b_i )$$

The following result is easy to check.

\begin{proposition}
The assignment ${\mathcal R} \mapsto {\mathcal R}_\oplus$ defines a functor from the category of topological ringoids and homomorphisms to the category of additive topological categories and additive functors.
\noproof
\end{proposition}

At times, it will be useful to look at ringoids with slightly more structure.

\begin{definition}
Let $R$ be a commutative topological ring equipped with an identity element.  Then we call a topological ringoid $\mathcal M$ an {\em $R$-moduloid} if:

\begin{itemize}

\item Each morphism set $\Hom (a,b)_{\mathcal M}$ is a topological $R$-module.

\item Given an element $r\in R$, and morphisms $x\in \Hom (a,b)_{\mathcal M}$ and $y\in \Hom (b,c)_{\mathcal M}$, we have the equation
$$r(xy) = (rx)y$$

\end{itemize}

\end{definition}

Observe that any ringoid can be considered a $\mathbb Z$-moduloid.  For convenience, all topological rings in this article are assumed to be commutative and unital.

\begin{definition}
Let $R$ be a topological ring.  A {\em non-unital topological moduloid}, $\mathcal M$, consists of a set of objects $\Ob ({\mathcal M})$, and a topological $R$-module, $\Hom (a,b)_{\mathcal M}$, for each pair of objects $a,b\in \Ob ({\mathcal M})$ such that the following axioms are satisfied:

\begin{itemize}

\item There is a continuous associative {\em composition law}
$$\Hom (b,c)_{\mathcal M}\times \Hom (a,b)_{\mathcal M} \rightarrow \Hom (a,c)_{\mathcal M}$$

\item Given an element $r\in R$, and morphisms $x\in \Hom (a,b)_{\mathcal M}$ and $y\in \Hom (b,c)_{\mathcal M}$, we have the equation

\end{itemize}

\end{definition}

Thus, a non-unital $R$-moduloid is a collection of objects and morphisms similar to an $R$-moduloid, apart from the fact that identity endomorphisms need not exist.

In the rest of this article, when we refer to an $R$-moduloid, it may by non-unital.  We will refer to a {\em unital} $R$-moduloid or {\em non-unital} $R$-moduloid when we wish to be more specific.

There are obvious notions of a homomorphism of $R$-moduloids, and of a sub-moduloid of an $R$-moduloid.  When necessary, we will use these notions without further comment.

\section{$K$-theory of Additive Categories}

It is well-known (see for example \cite{Seg2, Wa1}) how to define the $K$-theory of discrete additive categories.  The same approach can be used when there is a topology involved.

\begin{definition}
Let $\mathcal R$ be an additive topological category.  Then we define $N_n {\mathcal R}$ to be the category in which the objects are $n$-tuples of the form
$$(a_1 , \ldots , a_n ) \qquad a_i \in \Ob ({\mathcal R})$$
together with choices of biproduct $a_{i_1} \oplus \cdots \oplus a_{i_k}$ whenever $\{ i_1 , \ldots ,i_k \} \subseteq \{ 1, \ldots ,n \}$.

The morphisms and topology in the category $N_n {\mathcal R}$ are those arising from the morphisms and topology in the category ${\mathcal R}^n$.
\end{definition}

Let $N_\bullet {\mathcal R}$ denote the sequence of topological categories $(N_n {\mathcal R})$.  Then the sequence $N_\bullet {\mathcal R}$ can be considered a simplicial topological category.  We have face maps $\sigma_i \colon N_n {\mathcal R}\rightarrow N_{n-1}{\mathcal R}$ defined by writing
$$\sigma_i (a_1 , \ldots , a_n ) = (a_1 , \ldots , a_{i-1} , a_i \oplus a_{i+1}, a_{i+2} ,\ldots ,a_n ) \qquad i\neq 0$$
and
$$\sigma_0 (a_1 , \ldots , a_n ) = (a_2 ,\ldots ,a_n ) \qquad i\neq 0$$

We define degeneracy maps $\tau_i \colon N_n {\mathcal R}\rightarrow N_{n+1} {\mathcal R}$ are by writing
$$\tau_i (a_1 ,\ldots , a_n ) = (a_i ,\ldots , a_i , 0 , a_{i+1} ,\ldots ,a_n)$$
Note that the above formulae clearly determine what happens to morphisms and choices of biproducts.  It is easy to see that the category $N_\bullet {\mathcal R}$ is a simplicial additive topological category.  We can therefore iterate the above construction to obtain an $n$-simplicial additive topological category $N_\bullet^{(n)}{\mathcal R} := N_\bullet \cdots N_\bullet {\mathcal R}$.  Let us write $wN_\bullet^{(n)}{\mathcal R}$ to denote the category of isomorphisms in this category.

\begin{definition}
We define the {\em $K$-theory spectrum}, ${\mathbb K}({\mathcal R})$, of the additive topological category $\mathcal R$ to be the spectrum where the space ${\mathbb K}({\mathcal R})_n$ is the geometric realisation $|wN_\bullet^{(n)}{\mathcal R}|$ whenever $n\geq 1$.

There is a map $\Sigma |wN_\bullet^{(n)}{\mathcal R}| \rightarrow |wN_\bullet^{(n+1)}{\mathcal R}|$ defined by inclusion of the $1$-skeleton.  The structure map$${\mathbb K}({\mathcal R})_n \rightarrow \Omega {\mathbb K}({\mathcal R})_{n+1}$$
is defined to be the adjoint of the above inclusion.
\end{definition}

Note that we are using here the `thick' geometric realisation from appendix A of \cite{Seg2}.  When dealing with simplicial spaces rather than simplicial sets.  When dealing with simplicial spaces rather than just simplicial sets, the thick geometric realisation has appropriate formal properties involving homotopies and fibrations.

\begin{proposition}
The $K$-theory spectrum ${\mathcal R}\mapsto {\mathbb K}({\mathcal R})$ defines a functor from the category of additive topological categories to the category of symmetric spectra.
\end{proposition}

\begin{proof}
There is an action of the symmetric group $\Sigma_n$ on the simplicial category $wN_\bullet^{(n)}{\mathcal R} = wN_\bullet \cdots N_\bullet {\mathcal R}$ defined by permuting the order in which the $N_\bullet$ constructions are made.  There is therefore an induced $\Sigma_n$-action on the geometric realisation $|wN_\bullet^{(n)}{\mathcal R}|$.

The iterated structure map $|wN_\bullet^{(n)}{\mathcal R}|\rightarrow \Omega^k |wN_\bullet^{(n+k)} {\mathcal R}|$ is $\Sigma_k \times \Sigma_n$-equivariant in the obvious sense.  The spectrum ${\mathbb K}({\mathcal R})$ is therefore a {\em symmetric spectrum} as defined in \cite{HSS}.

Functoriality is straightforward to check.
\end{proof}

The paper \cite{Wa1} contains a construction of $K$-theory for discrete categories equipped with certain subcategories, called {\em categories of cofibrations} and {\em categories of weak equivalences}, that satisfy certain axioms.  It is shown in section 1.8 of \cite{Wa1} that this general construction of $K$-theory agrees with the above construction when we are looking at a discrete additive category.

The proof we mention still works for topological categories, although a topological version of Quillen's theorem A from \cite{Q1} is needed.  A suitable generalisation can be found, for example, in section 4 of \cite{Wa3}.

Many other properties of the $K$-theory of additive topological categories (or, more generally, categories with cofibrations and weak equivalences) can be proven exactly as in the discrete case.  In particular, the proof of the additivity theorem in \cite{McC} works without any modification at all, and many formal properties of $K$-theory can be obtained as corollaries of the additivity theorem; see \cite{Sta} for details of the argument.

The following result can be proved in this way, as can the fibration theorem and cofinality theorem we state later on in this section.

\begin{theorem}
Let $\mathcal R$ be a topological ringoid.  Then the structure map
$${\mathbb K}({\mathcal R})_n \rightarrow \Omega {\mathbb K}({\mathcal R})_{n+1}$$
is a weak homotopy equivalence whenever $n\geq 1$.
\noproof
\end{theorem}

Whenever $n\geq 1$, we can define {\em $K$-theory groups}
$$K_n ({\mathcal R}):= \pi_n {\mathbb K}({\mathcal R})$$

It follows from the above result that these groups can be defined as ordinary homotopy groups
$$K_n ({\mathcal R}) = \pi_n \Omega |wN_\bullet {\mathcal R}_\oplus |$$
rather than as the stable homotopy groups of some spectrum.

\begin{definition}
Let $f\colon {\mathcal A}\rightarrow {\mathcal B}$ be a continuous additive functor.  Let $PN_\bullet {\mathcal B}$ be the simplicial category defined by writing $(PN_\bullet {\mathcal B})_n = N_{n+1}{\mathcal B}$.  Then we define the simplicial additive topological category $N_\bullet (f\colon {\mathcal A}\rightarrow {\mathcal B})$ by the pullback diagram
$$\begin{array}{ccc}
N_\bullet (f\colon {\mathcal A}\rightarrow {\mathcal B}) & \rightarrow & PN_\bullet {\mathcal B} \\
\downarrow & & \downarrow \\
N_\bullet {\mathcal A} & \rightarrow & N_\bullet {\mathcal B} \\
\end{array}$$
\end{definition}

The simplicial category $PN_\bullet {\mathcal B}$ is clearly additive.  It is straightforward to show that the pullback $N_\bullet (f\colon {\mathcal A}\rightarrow {\mathcal B})$ is also additive.  We can therefore define spaces $|wN_\bullet^{(n)} (f\colon {\mathcal A}\rightarrow {\mathcal B})$ and a $K$-theory spectrum
$${\mathbb K}(f\colon {\mathcal A}\rightarrow {\mathcal B})$$

For obvious reasons, we will call the following result the {\em fibration theorem}.

\begin{theorem} \label{fibration}
The canonical sequence
$${\mathbb K}(f\colon {\mathcal A}\rightarrow {\mathcal B}) \rightarrow {\mathbb K}({\mathcal A})\rightarrow {\mathbb K}({\mathcal B})$$
is a fibration up to homotopy.
\noproof
\end{theorem}

Let ${\mathcal A}$ be an additive subcategory of some additive category $\mathcal B$.  Then we say that the category $\mathcal A$ is {\em strictly cofinal} in $\mathcal B$ if for each object $b\in \Ob ({\mathcal B})$, there is an object $a\in \Ob ({\mathcal A})$ such that the biproduct $b\oplus a$ is isomorphic to some object in the category ${\mathcal A}$.

The following result is known as the {\em cofinality theorem}, and is also a corollary of the additivity theorem.

\begin{theorem} \label{cofinality}
Let $\mathcal B$ be an additive topological category.  The $\mathcal A$ be a strictly cofinal subcategory of $\mathcal B$.  Then the inclusion ${\mathcal A}\hookrightarrow {\mathcal B}$ induces a homotopy equivalence of $K$-theory spectra.
\noproof
\end{theorem}

\section{$K$-theory of Ringoids and Moduloids}  \label{Kring}

\begin{definition}
Let $\mathcal R$ be a topological ringoid.  Then we define the $K$-theory of $\mathcal R$ by the formula
$${\mathbb K}({\mathcal R}):={\mathbb K}({\mathcal R}_\oplus )$$
\end{definition}

If the ringoid $\mathcal R$ is already additive, then it is clearly cofinal in the category ${\mathcal R}_\oplus$.  Therefore, by theorem \ref{cofinality} the above definition is consistent, at least up to homotopy, with the old definition made for additive topological categories.

Given a topological ring $R$, a unital $R$-moduloid $\mathcal M$ is a topological ringoid, so we can define the $K$-theory ${\mathbb K}({\mathcal M})$ by the above procedure.  We need to do slightly more in the non-unital case, however.

\begin{definition}
Let $\mathcal M$ be a non-unital moduloid over a topological ring $R$.  Then we define the {\em unitisation} ${\mathcal M}^+$ to be the category with the same objects as the moduloid $\mathcal M$, and with morphism sets
$$\Hom (a,b)_{{\mathcal M}^+} = \left\{ \begin{array}{ll}
\Hom (a,b)_{\mathcal M} & a\neq b \\
\Hom (a,a) \oplus R & a=b \\
\end{array} \right.$$

The topology on the morphism sets is either the original topology or the product topology, depending on which of the above cases we are considering.  Composition of morphisms
$$\Hom (b,c)_{{\mathcal M}^+} \times \Hom (a,b)_{{\mathcal M}^+} \rightarrow \Hom (a,c)_{{\mathcal M}^+}$$
is defined by the formula
$$(x+\lambda)(y+\mu) = xy +\ lambda x + \mu y + \lambda \mu$$
where $x\in \Hom (b,c)_{\mathcal M}$, $y\in \Hom (a,b)_{\mathcal M}$, and
$$\lambda \in \left\{ \begin{array}{ll}
R & b=c \\
\{ 0 \} b\neq c \\
\end{array} \right. \qquad \mu \in \left\{ \begin{array}{ll}
R & a=b \\
\{ 0 \} a\neq b \\
\end{array} \right.$$
\end{definition}

It is easy to check that the unitisation ${\mathcal M}^+$ is a unital $R$-moduloid.

Let us write $R_{\mathcal M}$ to denote the unital $R$-moduloid with the same set of objects as the moduloid $\mathcal M$, and morphism sets
$$\Hom (a,b)_{R_{\mathcal M}} = \left\{ \begin{array}{ll}
R & a=b \\
\{ 0 \} & a\neq b \\
\end{array} \right.$$
equipped with the obvious multiplication and topology.  There is a natural homomorphism $\pi \colon {\mathcal M}^+ \rightarrow R_{\mathcal M}$ defined to be the identity map on the set of objects, and by the formula
$$\pi (x + \lambda ) = \lambda \qquad x\in \Hom (a,b)_{\mathcal M}, \lambda \in R$$
on each set of objects.

\begin{definition}
We define the $K$-theory spectrum  of a non-unital moduloid $\mathcal M$ be the formula
$${\mathbb K}({\mathcal M}) := {\mathbb K}(\pi \colon {\mathcal M}^+ \rightarrow R_{\mathcal M})$$
\end{definition}

According to theorem \ref{fibration}, the spectrum ${\mathbb K}({\mathcal M})$ is be the homotopy fibre of the map
$$\pi_\star \colon {\mathbb K}({\mathcal M}^+ ) \rightarrow {\mathbb K}(R_{\mathcal M})$$

Before we go any further with our explorations, we need to prove that the above definition is consistent with the old definition when the $R$-moduloid $\mathcal M$ is unital.

\begin{proposition}
Let $\mathcal M$ be a unital $R$-moduloid.  Let $\pi' \colon {\mathcal M}\oplus R_{\mathcal M}\rightarrow R_{\mathcal M}$ be the obvious quotient map.  Then there is a natural isomorphism $\alpha \colon {\mathcal M}\oplus R_{\mathcal M} \rightarrow {\mathcal M}^+$ fitting into a commutative diagra
m$$\begin{array}{ccc}
{\mathcal M}\oplus R_{\mathcal M} & \stackrel{\pi'}{\rightarrow} & R_{\mathcal M} \\
\downarrow & &  \| \\
{\mathcal M}^+ & \stackrel{\pi}{\rightarrow} & R_{\mathcal M} \\
\end{array}$$
\end{proposition}

\begin{proof}
Let us write $e_a \in \Hom (a,a)_{\mathcal M}$ to denote an identity morphism in the moduloid $\mathcal M$.  We can define a natural isomorphism $\alpha \colon {\mathcal M}\oplus R_{\mathcal M} \rightarrow {\mathcal M}^+$ by writing $\alpha (a) =a$ for each object $a$, and
$$\alpha (x ,\lambda ) = x - \lambda e_a + \lambda$$
whenever $(x,\lambda ) \in \Hom (a,b)_{{\mathcal M}\oplus R_{\mathcal M}}$.

The inverse of the homomorphism $\alpha$ is defined by the formula
$$\alpha^{-1} (y + \mu ) = (y + \mu e_a ,\mu )$$
whenever $y+ \mu \in \Hom (a,b)_{{\mathcal M}^+}$.
\end{proof}

\begin{corollary}
Let $\mathcal M$ be a unital moduloid.  Then the $K$-theory spectra ${\mathbb K}({\mathcal M})$ and ${\mathbb K}(\pi \colon {\mathcal M}^+ \rightarrow R_{\mathcal M})$ are naturally homotopy-equivalent.
\end{corollary}

\begin{proof}
By the above proposition and theorem \ref{fibration} we have a commutative diagram
$$\begin{array}{ccccc}
{\mathbb K}({\mathcal M}) & \rightarrow & {\mathbb K}({\mathcal M}\oplus R_{\mathcal M}) & \rightarrow & {\mathbb K}(R_{\mathcal M}) \\
\downarrow & & \downarrow & &  \| \\
{\mathbb K}(\pi \colon {\mathcal M}^+ \rightarrow R_{\mathcal M}) & {\mathbb K}({\mathcal M}^+) & \rightarrow & {\mathbb K}(R_{\mathcal M}) \\
\end{array}$$
where the rows are homotopy fibrations.  The desired result follows easily.
\end{proof}

\section{Quotients and Exact Sequences}

\begin{definition}
Let $R$ be a topological ring, and let $\mathcal M$ be an $R$-moduloid.  Let $\mathcal J$ be a sub-moduloid with the same set of objects as the moduloid $\mathcal M$.  Then we call the sub-moduloid $\mathcal J$ an {\em ideal} in the moduloid $\mathcal M$ if the product of a morphism from the moduloid $\mathcal J$ and a morphism from the moduloid $\mathcal M$ is always a morphism in the moduloid $\mathcal J$.
\end{definition}

\begin{definition}
Let $\mathcal J$ be an ideal in an $R$-moduloid $\mathcal M$.  Then we define the {\em quotient} ${\mathcal M}/{\mathcal J}$ to be the $R$-moduloid with the same objects as the moduloid $\mathcal M$, with morphism sets
$$\Hom (a,b)_{{\mathcal M}/{\mathcal J}} :=\Hom (a,b)_{{\mathcal M}/{\mathcal J}} /\Hom (a,b)_{\mathcal J}$$
\end{definition}

It is straightforward to check that the above definition makes sense, and does indeed define an $R$-moduloid as inherently claimed.  There is a canonical inclusion map $i\colon {\mathcal J} \rightarrow {\mathcal M}$ and quotient map $j \colon {\mathcal M}\rightarrow {\mathcal M}/{\mathcal J}$.

\begin{theorem}
The sequence
$${\mathbb K}({\mathcal J}) \stackrel{i_\star}{\rightarrow} {\mathbb K}({\mathcal M}) \stackrel{j_\star}{\rightarrow} {\mathbb K}({\mathcal M}/{\mathcal J})$$is a fibration up to homotopy.
\end{theorem}

\begin{proof}
Suppose that the moduloid $\mathcal M$ is unital.  Recall that the spectrum ${\mathbb K}({\mathcal J})$ is defined by iterating the simplicial additive category $N_\bullet (\pi \colon {\mathcal J}^+ \rightarrow R_{\mathcal M})$ defined through the pull-back diagram
$$\begin{array}{ccc}
N_\bullet (\pi \colon {\mathcal J}^+\rightarrow R_{\mathcal M}) & \rightarrow & PN_\bullet {\mathcal R_{\mathcal M}} \\
\downarrow & & \downarrow \\
N_\bullet {\mathcal J}^+ & \rightarrow & N_\bullet R_{\mathcal M} \\
\end{array}$$

Since the moduloid $\mathcal M$ is unitial, there is an obvious inclusion ${\mathcal J}^+\hookrightarrow {\mathcal M}$ and a commutative diagram
$$\begin{array}{ccc}
N_\bullet (\pi \colon {\mathcal J}^+\rightarrow R_{\mathcal M}) & \rightarrow & PN_\bullet {\mathcal R_{\mathcal M}} \\
\downarrow & & \downarrow \\
N_\bullet {\mathcal J}^+ & \rightarrow & N_\bullet R_{\mathcal M} \\
\downarrow & & \downarrow \\
N_\bullet {\mathcal M} & \rightarrow & {\mathcal M}/{\mathcal J} \\
\end{array}$$

The outer square of this diagram is still a pull-back square.  It follows that we have an isomorphism
$$N_\bullet (\pi \colon {\mathcal J}^+\rightarrow R_{\mathcal M}) \cong
N_\bullet (j \colon {\mathcal M}\rightarrow {\mathcal M}/{\mathcal J})$$

It follows from theorem \ref{fibration} that we thus have a fibration
$${\mathbb K}({\mathcal J}) \stackrel{i_\star}{\rightarrow} {\mathbb K}({\mathcal M}) \stackrel{j_\star}{\rightarrow} {\mathbb K}({\mathcal M}/{\mathcal J})$$is a fibration up to homotopy.
up to homotopy.

Now, suppose that the moduloid $\mathcal M$ is not unital.  Then by the above argument we have a commutative diagram
$$\begin{array}{ccccc}
{\mathbb K}({\mathcal J}) & \rightarrow & {\mathbb K}({\mathcal M}) & \rightarrow & {\mathbb K}({\mathcal M}/{\mathcal J}) \\
\| & & \downarrow & & \downarrow \\
{\mathbb K}({\mathcal J}) & \rightarrow & {\mathbb K}({\mathcal M}^+) & \rightarrow & {\mathbb K}({\mathcal M}^+/{\mathcal J}) \\
\downarrow & & \downarrow & & \downarrow \\
\ast & \rightarrow & {\mathbb K}(R_{\mathcal M}) & = & {\mathbb K}(R_{\mathcal M}) \\
\end{array}$$
where the columns and bottom two rows are fibrations up to homotopy.  A standard diagram chase tells us that the top row is also a fibration up to homotopy, and we are done.
\end{proof}

\section{Products}

Tensor products of moduloids can be defined analogously to tensor products of modules over commutative rings.

\begin{definition}
Let $R$ be a commutative unital topological ring.  Let $\mathcal M$ and $\mathcal N$ be $R$-moduloids.  Then we define the {\em tensor product}, ${\mathcal M}\otimes_R {\mathcal N}$, of the moduloids $\mathcal M$ and $\mathcal N$ to be the collection of objects
$$\{ a\otimes b\ |\ a\in \Ob ({\mathcal M}) , b\in \Ob ({\mathcal N}) \}$$
and morphisms
$$\Hom (a\otimes b , a'\otimes b' )_{{\mathcal M}\otimes {\mathcal N}} = \Hom (a,a')_{\mathcal M}\otimes_R \Hom (b,b')_{\mathcal N}$$
\end{definition}

If $\mathcal R$ and $\mathcal S$ are ringoids, then we define the tensor product by writing ${\mathcal R}\otimes {\mathcal S} := {\mathcal R}\otimes_{\mathbb Z} {\mathcal S}$

The following result is straightforward to check.

\begin{proposition}
The above category ${\mathcal M}\otimes_R {\mathcal N}$ is an $R$-moduloid.  The composition law is defined by the formula
$$(x\otimes y)(x'\otimes y') = xx' \otimes yy'$$
\noproof
\end{proposition}

In fact, the above tensor product also satisfies an appropriate universal property.  Although this property is easy to verify, we will not prove it here since we do not need to use it.  Observe, however, that for $R$-moduloids $\mathcal M$ and $\mathcal N$ we have a canonical map
$${\mathcal M}_\oplus \times {\mathcal N}_\oplus \rightarrow ({\mathcal M}\otimes_R{\mathcal N})_\oplus$$
defined by writing
$$(a_1 \oplus \cdots \oplus a_m, b_1 \oplus \cdots \oplus b_n) \mapsto \oplus_{i,j=1}^{m,n} a_i\otimes b_j$$

The above map clearly induces a $\Sigma_k \times \Sigma_n$-equivariant map
$$wN_\bullet^{(k)}{\mathcal M}_\oplus \times wN_\bullet^{(n)}{\mathcal N}_\oplus \rightarrow wN_\bullet^{(n+k)}({\mathcal M}\otimes_R{\mathcal N})_\oplus$$
and so a natural map of symmetric spectra
$${\mathbb K}({\mathcal M})\wedge {\mathbb K}({\mathcal N})\rightarrow {\mathbb K}({\mathcal M}\otimes_R {\mathcal N})$$

\begin{definition}
The above map
$${\mathbb K}({\mathcal M})\wedge {\mathbb K}({\mathcal N})\rightarrow {\mathbb K}({\mathcal M}\otimes_R {\mathcal N})$$
is called the {\em exterior product map}.
\end{definition}

\section{Banach Categories and Group Completion}

Recall that a topological ringoid $\mathcal A$ is called a {\em Banach category} if each morphism set $\Hom (a,b)_{\mathcal A}$ is a Banach space, and the inequality
$$\| xy \| \leq \| x\| \| y\|$$
holds whenever $x$ and $y$ are composable morphisms.  Since a Banach category $\mathcal A$ is already a topological ringoid, we can use our earlier definitions to form the $K$-theory spectrum ${\mathbb K}({\mathcal A})$.

However, the papers \cite{Jo2,Mitch2.5} define $K$-theory for $C^\ast$-categories, of which Banach categories are a special case.  The definitions in either of these papers could easily be generalised to Banach categories.  The definition amounts to the following.

\begin{definition}
Let $\mathcal A$ be a Banach category.  Fix an ordering on the set of objects $\Ob ({\mathcal A})$.  For each object $a\in \Ob ({\mathcal A}_\oplus )$, let $GL(A)$ denote the space of invertible endomorphisms $x\in \Hom (a,a)_{{\mathcal A}_\oplus }$.

We define the topological group $GL_\infty ({\mathcal A})$ to be the direct limit
$$\bigcup_{a_1 \leq \cdots \leq a_n} GL (a_1 \oplus \cdots \oplus a_n )$$
under the inclusions
$$GL( a_1 \oplus \cdots \oplus a_{k-1}\oplus a_{k+1} \oplus \cdots \oplus a_n ) \hookrightarrow GL( a_1 \oplus \cdots \oplus a_n )$$
defined in the obvious way by looking at matrices.
\end{definition}

The space $GL_\infty ({\mathcal A})$ is not functorial.  However, it can if necessary be replaced by a homotopy-equivalent space that is functorial.   We do not need to do this here; see \cite{Mitch2.5} for details.

\begin{definition}
We define the {\em analytic $K$-theory groups} of the Banach category $\mathcal A$ to be the homotopy groups
$$K_n ({\mathcal A})^\mathrm{Analytic} := \pi_n BGL_\infty ({\mathcal A})$$
when $n\geq 1$.
\end{definition}

In contrast, the {\em algebraic $K$-theory groups} of a topological ringoid $\mathcal R$ are defined by writing
$$K_n ({\mathcal A}) := \pi_n {\mathbb K}({\mathcal A}) = \pi_n \Omega |wN_\bullet {\mathcal A}_\oplus |$$
when $n\geq 1$.

The main purpose of this section is to prove the fact that the algebraic and analytic $K$-theory groups of a Banach category are isomorphic.

The argument required is an application of the group completion theorem.  The version given in \cite{McS} suffices for our purposes.

\begin{theorem}
Let $M$ be a topological monoid, and let $H_\star (M)$ be the associated Pontjagin rind built from the induced product on singular homology.  Suppose that the monoid $\pi = \pi_0 (M)$ is in the centre of the ring $H_\star (M)$.  Then the ring homomorphism $i_\star \colon H_\star (M)[\pi^{-1}] \rightarrow H_\star (\Omega BM)$ induced by the canonical map $i\colon M\rightarrow \Omega BM$ induces an isomorphism
$$H_\star (M)[\pi^{-1}] \rightarrow H_\star (\Omega BM)$$
\noproof
\end{theorem}

Along with the group completion theorem, we need a sufficiently general version of Whitehead's theorem.  The following result is proved in \cite{May4}.

\begin{theorem} \label{Whitehead}
Let $X$ and $Y$ be path-connected $H$-spaces.  Suppose we have a map of $H$-spaces $f\colon X\rightarrow Y$ such that the induced maps of homology groups $f_\star \colon H_p (X)\rightarrow H_p (Y)$ are all isomorphisms.  Then the map $f$ is a weak homotopy equivalence.
\noproof
\end{theorem}

Before applying the above two results, observe that we can simplify our calculations slightly if we observe that there are isomorphisms
$$K_n ({\mathcal A})^\mathrm{Analytic} \cong \pi_n BGL_\infty (\Sk ({\mathcal A}))$$
and
$$K_n ({\mathcal A}) \cong \pi_n \left( \Omega B \amalg_{a\in Sk ({\mathcal A}_\oplus )} BGL (a) \right)$$
where $\Sk ({\mathcal C})$ denotes the skeleton of a category $\mathcal C$.  Roughly speaking, our plan is to prove that the spaces $BGL_\infty (\Sk ({\mathcal A}))$ and $\Omega B \amalg_{a\in Sk ({\mathcal A}_\oplus )} BGL (a)$ are weakly homotopy-equivalent.

In order to apply the group completion theorem, we want to define a topological monoid structure on the space
$$M:=\amalg_{a\in \Sk ({\mathcal A}_\oplus )} BGL (a)$$

Any two objects $a,b\in \Sk ({\mathcal A}_\oplus )$ have a unique biproduct $a\oplus b$.  Let $i_a \colon a\rightarrow a\oplus b$, $i_b \colon b\rightarrow a\oplus b$, $p_a \colon a\oplus b\rightarrow a$, and $p_b \colon a\oplus b\rightarrow b$ be the canonical morphisms associated to this biproduct, and define maps $j_a \colon GL(a)\rightarrow GL(a\oplus b)$ and $j_b \colon GL(b)\rightarrow GL(a\oplus b)$ by the formulae
$$j_a (x) = i_a x p_a + i_b p_b \qquad j_b (y) = i_a p_a + i_b y p_b$$
repsectively.

After delooping, we have induced continuous maps $Bj_a \colon BGL(a)\rightarrow BGL(a\oplus b)$ and $Bj_b \colon BGL(b)\rightarrow BGL(a\oplus b)$.  The space $M$ is therefore a topological monoid, with operation
$$x\oplus y = Bj_a (x) + Bj_b (y)$$

Observe that the set $\pi = \Ob (\Sk ({\mathcal A}_\oplus ))$ can be identified with the set of path-components $\pi_0 (M)$.  Let $G$ be the Grothendieck completion of the monoid $\pi$.  Then there is a homomorphism of topological monoids
$$\alpha \colon M\rightarrow BGL_\infty (\Sk ({\mathcal A})) \times G$$
defined by taking a point $x\in BGL(a)$ to the pair $(i(x),a)$ where $i\colon BGL (a)\rightarrow BGL_\infty ({\mathcal A})$ is the canonical inclusion.

\begin{lemma} \label{main_lemma}
The above homomorphism induces a ring isomorphim
$$\alpha_\star \colon H_\star (M)[\pi^{-1}] \rightarrow H_\star (BGL_\infty (\Sk ({\mathcal A}) \times G)$$
\end{lemma}

We defer the proof of the above lemma.  Together with the group completion theorem, this lemma is the key to our main result.

\begin{theorem}
The spaces $BGL_\infty (\Sk ({\mathcal A}))\times G$ and $\Omega B \amalg_{a\in Sk ({\mathcal A}_\oplus )} BGL (a)$ are weakly homotopy-equivalent.
\end{theorem}

\begin{proof}
The monoid $\pi$ clearly lies is the centre of the Pontrjagin ring $H_\star (M)$.  The map $\alpha$ induces a map $\Omega BM \rightarrow \Omega B(BGL_\infty (\Sk ({\mathcal A}))\times G)$ of $H$-spaces.  But the space $BGL_\infty (\Sk ({\mathcal A}))\times G$ is already a topological group, meaning it is homotopy-equivalent to the $H$-space $\Omega B(BGL_\infty (\Sk ({\mathcal A}))\times G)$.  Hence, by the above lemma and the group completion theorem, the canonical map of $H$-spaces
$$\Omega B \amalg_{a\in \Sk ({\mathcal A}_\oplus )} BGL(a) \rightarrow BGL_\infty (\Sk ({\mathcal A}))\times G$$
induces an isomorphism at the level of homology.

Now, let $X_0$ denote the path-component of the $H$-space $\Omega B\amalg_{a\in \Sk ({\mathcal A}_\oplus )} BGL (a)$ containing the unit.  Then we have a map of path-connected $H$-spaces
$$X_0 \rightarrow BGL_\infty (\Sk ({\mathcal A})$$
that induces an isomorphism at the level of homology.  Hence, by theorem \ref{Whitehead}, the map $X_0 \rightarrow BGL_\infty (\Sk ({\mathcal A}))$ is a homotopy-equivalence.

Each path-component of the space $\Omega B \amalg_{a\in Sk ({\mathcal A}_\oplus )} BGL (a)$ is homotopy-equivalent to the space $X_0$.  Looking at each path-component in turn, it follows that we have a homotopy-equivalence
$$\Omega B \amalg_{a\in \Sk ({\mathcal A}_\oplus )} BGL(a) \rightarrow BGL_\infty (\Sk ({\mathcal A}))\times G$$
as required.
\end{proof}

Hence, by the definitions of the relevant $K$-theory groups we obtain the isomorphism
$$K_n ({\mathcal A})\cong K_n ({\mathcal A})^\mathrm{Analytic}$$
that we were looking for.

Of course, we still need to prove lemma \ref{main_lemma}.  We need two technical results about analytic $K$-theory.  The first of these results is a special case of theorem 3.1 in \cite{Mitch2.5}.  See also the appendix of \cite{Mil}.

\begin{proposition} \label{tp1}
Let $(X_i , \phi_{ij})_{i\in I}$ be a directed family of topological spaces such that each map $\phi_{ij} \colon X_i \rightarrow X_j$ is injective, each space $X_i$ has the $T_1$-separation property\footnote{That is to say that finite subsets are closed.}, and for each element $j\in I$ there are only finitely many elements $i\in I$ such that $i\leq j$.

Let $X$ be the direct limit of this family, and let $\phi_i \colon X_i \rightarrow X$ be the canonical map associated to each space $X_i$.  Suppose that $K\subseteq X$ is a compact set.  Then $K\subseteq \phi_i [X_i]$ for some space $X_i$.
\noproof
\end{proposition}

The second technical proposition follows from commutativity of the analytic $K$-theory groups, a result again proved in \cite{Mitch2.5}.

\begin{proposition} \label{tp2}
The topological monoid $BGL_\infty (Sk ({\mathcal A}))$ is homotopy-commutative.
\noproof
\end{proposition}

Let $\sigma$ be a permutation of the set $\{ 1,\ldots , k \}$.  It follows from the above results that any map $X\rightarrow BGL(a_1 \oplus \cdots \oplus a_k )$ is homotopic to a map $X\rightarrow BGL(\sigma (a_1)\oplus \cdots \oplus \sigma (a_k))$.  With this observation, we can proceed to the proof of our lemma.

\begin{lemma}
The homomorphism $\alpha \colon M\rightarrow BGL_\infty (\Sk ({\mathcal A})) \times G$ induces a ring isomorphim
$$\alpha_\star \colon H_\star (M)[\pi^{-1}] \rightarrow H_\star (BGL_\infty (\Sk ({\mathcal A}) \times G)$$
\end{lemma}

\begin{proof}
Let $[\sigma ] \in H_k (BGL_\infty (\Sk ({\mathcal A}))\times G)$ be a singular homology class, generated by a continuous map
$$\sigma \colon \Delta^k \rightarrow BGL_\infty (\Sk ({\mathcal A}))\times G$$

By proposition \ref{tp1}, we know that $\im \sigma \subseteq BGL (a)\times \{ g \}$ for some object $a\in \Ob (\Sk ({\mathcal A})_\oplus )$ and element $g\in G$.  By proposition \ref{tp2} , we can assume that $a\in \Ob (\Sk ({\mathcal A}_\oplus ))$.  Let us write $g=a\oplus b\oplus c^{-1}$ where $b,c\in \pi$.  Then certainly
$$\im \sigma \subseteq BGL (a\oplus b) \times \{ a\oplus b\oplus c^{-1} \}$$

Hence, by definition of the map $\alpha$, the homology class $[\sigma ]$ lies in the image of the homomorphism
$$\alpha_\star \colon H_\star (m)[\pi^{-1}] \rightarrow H_\star (BGL_\infty (\Sk ({\mathcal A}))\times G)$$

Thus the homomorphism $\alpha_\star$ is surjective.  Injectivity can be proven similarly.
\end{proof}

\section{Assembly}

The assembly map in algebraic $K$-theory originally appeared in \cite{Lod}.  A version of this map, adapted to give a homotopy-theoretic description of the analytic assembly map, was used in \cite{Ti1}.

In this section we will describe the $K$-theory assembly map in terms of the tools developed here, and give a new proof that the homotopy-theoretic description of the analytic assembly map agrees with the usual definition.

\begin{definition}
Let $\pi$ be a discrete groupoid, and let $R$ be a topological ring.  Then we define $R\pi$ to be the category with objects $\Ob (R\pi ) = \Ob (\pi )$ and morphism sets
$$\Hom (a,b)_{R \pi} := \{ x_1 g_1 + \cdots + x_n g_n \ |\ x_i \in R, g_i \in \Hom (a,b)_\pi \}$$

Composition of morphisms is defined by the formula
$$\left( \sum_{i=1}^m x_i g_i \right) \left( \sum_{j=1}^n y_j h_j \right) = \sum_{i,j=1}^{m,n} (x_i y_j ) g_i h_j$$

The topology on the space $\Hom (a,b)_{R \pi}$ is defined by viewing it as a direct limit of a family of subspaces homeomorphic to the space $R^n$ for some $n$.\end{definition}

It is easy to check that the category $R\pi$ is a topological $R$-moduloid.

\begin{proposition}
Let $R$ be a topological ring.  Then the category $R\pi$ is isomorphic to the tensor product ${\mathbb Z}\pi \otimes_{\mathbb Z} R$.
\end{proposition}

\begin{proof}
We can define an isomorphism $\theta \colon R\pi \rightarrow \pi {\mathbb Z}\otimes_{\mathbb Z} R$ by writing $\theta (A) =A$ for each object $A$, and
$$\theta (x_1 g_1 + \cdots +x_n g_n ) = g_1\otimes x_1 + \cdots + g_n \otimes x_n$$
for all elements $x_i \in R$ and $g_i \in \Hom (A,B)_\pi$.
\end{proof}

Now, we have an inclusion
$$\Sigma | w{\mathbb Z}\pi_\oplus | \hookrightarrow |wN_\bullet {\mathbb Z}\pi_\oplus |$$
defined by looking at the $1$-skeleton, and therefore, taking adjoints, a natural map
$$| w{\mathbb Z}\pi_\oplus | \rightarrow \Omega |wN_\bullet {\mathbb Z}\pi_\oplus |$$

Composing this map with the inclusion $B\pi \hookrightarrow \Sigma |w{\mathbb z}\pi_\oplus |$, we have a natural map
$$B\pi_+ \rightarrow \Omega |wN_\bullet {\mathbb Z}\pi_\oplus |$$

Here we write $X_+$ to denote a topological space with an added disjoint basepoint.  If we do this, we can deem the above map to be basepoint-preserving.  We can compose the above map with the exterior product to obtain a natural map
$$B\pi_+\wedge {\mathbb K}(R) \rightarrow {\mathbb K}({\mathbb Z}\pi ) \wedge {\mathbb K}(R) \rightarrow {\mathbb K}(R\pi )$$

\begin{definition}
The above map is called the {\em $K$-theory assembly map} associated to the groupoid $\pi$ and topological ring $R$.
\end{definition}

In order to compare definitions of assembly maps, we need an abstract characterisation.  One suitable characterisation can be found in \cite{WW}.

\begin{definition}
Let $F$ be a covariant functor from the category of spaces to the stable category of symmetric spectra.  Then we call the functor $F$

\begin{itemize}

\item {\em homotopy-invariant} if it preserves homotopy-equivalences.

\item {\em excisive} if it preserves coproducts up to homotopy equivalence.

\end{itemize}

\end{definition}

Let $+$ denote the one point topological space.

\begin{theorem}
Let $F$ be a homotopy-invariant functor from the category of spaces to the stable category of symmetric spectra.  Then there is a homotopy-invariant excisive functor $F^\%$ and a natural transformation $\alpha \colon F^\% \rightarrow F$ such that the map $\alpha \colon F^\% (+) \rightarrow F(+)$ induces an isomorphism at the level of stable homotopy groups.

Further, the functor $F^\%$ and natural transformation $\alpha$ are unique up to stable equivalence.
\noproof
\end{theorem}

We call the natural transformation $\alpha \colon F^\% \rightarrow F$ the {\em assembly map} associated to the functor $F$.

\begin{theorem} \label{clasKt}
Let $X$ be a topological space, and let $\pi (X)$ be the fundamental groupoid of $X$, equipped with the discrete topologicy.  Write
$$F(X) = {\mathbb K}(R\pi (X)) \qquad F^\% (X) = X_+ \wedge {\mathbb K}(R)$$

Then there is a natural transformation $\alpha \colon F^\% \rightarrow F$
such that the map
$$\alpha \colon F^\% (B\pi ) \rightarrow F( \pi )$$
is homotopy-equivalent to the $K$-theory assembly map for a discrete groupoid $\pi$.

Further, suppose we have a natural transformation $\alpha' \colon F^\% \rightarrow F$ such that the map $\alpha' \colon F^\% (+) \rightarrow F(+)$ is a weak homotopy-equivalence.  Then the natural transformation $\alpha'$ is weakly homotopic to the map $\alpha$.
\end{theorem}

\begin{proof}
Let $\pi^\tau (X)$ denote the fundamental groupoid of the space $X$, equipped with the topologiy defined by considering the groupoid $\pi^\tau (X)$ to be a quotient of the space of Moore paths
$$\{ \gamma \colon [0,s]\rightarrow X \ | \ s\geq 0 \}$$
equipped with the compact open topology.  Then the identity map defines a continuous functor $\pi (X)\rightarrow \pi^\tau (X)$.  It is proved in the appendix by J.P.~May to \cite{RW} that the induced map of classifying spaces
$$f\colon B\pi (X) \rightarrow B\pi^\tau (X)$$
is a homotopy-equivalence.  Hence the induced map of spectra $f_\star \colon F^\% (B\pi (X)) \rightarrow F^\% (B\pi^\tau (X))$ has a formal inverse $f^\star$ in the category of symmetric spectra.

There is a natural inclusion $X\rightarrow B\pi^\tau (X)$ defined by mapping a point $x\in X$ to the point corresponding to the identity $1_x$.  At the level of spectra we obtain a natural map $\beta$ by forming the composition
$$F^\% (X) \rightarrow F^\% (B\pi^\tau (X)) \stackrel{f^\star}{\rightarrow} F^\% (B\pi (X))$$

Composing the map $\beta$ with the $K$-theory assembly map, we obtain a natural map $\alpha \colon F^\%(X) \rightarrow F (X)$, which is a homotopy-equivalence when the space $X$ is a single point.  Let $\pi$ be a discrete groupoid, and let $X=B\pi$.  Then the map $\alpha$ is homotopic to the $K$-theory assembly map.

Finally, note that it is easy to see that the functor $F$ is homotopy-invariant, and the functor $F^\%$ is homotopy-invariant and excisive.  Therefore the map $\alpha$ is an assembly map in the sense of the above theorem.  Hence, if we have a natural transformation $\alpha' \colon F^\% \rightarrow F$ such that the map $\alpha' \colon F^\% (+) \rightarrow F(+)$ is a weak homotopy-equivalence, then the natural transformation $\alpha'$ is weakly homotopic to the map $\alpha$.
\end{proof}

We will end this section by looking at the assembly map appearing in the analytic Novikov conjecture, which we shall refer to as the {\em analytic assembly map}.  To fit this map into our present framework, recall the following result from \cite{Mitch3}, which provides all we need to know about the analytic assembly map for the present purposes.

\begin{theorem} \label{mitch3}
Let $F^\%$ be a homotopy-invariant and excisive functor from the category of spaces to the stable category of symmetric spectra.  Let $\alpha \colon F^\% (X)\rightarrow {\mathbb K}C^\ast_\mathrm{max} \pi (X)$ be a natural map of spectra that is a stable equivalence when the space $X$ is a single point.  Let $G$ be a discrete group.  Then the map
$$\alpha \colon F^\% (BG )\rightarrow {\mathbb K}C^\ast_\mathrm{max} \pi (BG) \cong {\mathbb K} C^\ast_\mathrm{max} (G )$$
is a description of the analytic assembly map for the group $G$ at the level of connective spectra.
\noproof
\end{theorem}

Here, for a discrete groupoid $\pi$, the category $C^\ast \pi$ is a certain Banach category obtained by completing the topological ringoid ${\mathbb C}\pi$; see \cite{Mitch2} for further details.  There is a canonical inclusion ${\mathbb C}\pi \hookrightarrow C^\ast \pi$, and so a natural map ${\mathbb K} {\mathbb C}\pi \rightarrow {\mathbb K} C^\ast_\mathrm{max} \pi$.

\begin{theorem}
Let $G$ be a discrete group.  The composition of the $K$-theory assembly map and the above map
$$BG _+ \wedge {\mathbb K}({\mathbb C}) \rightarrow {\mathbb K} {\mathbb C}G \rightarrow {\mathbb K} C^\ast_\mathrm{max} G$$
is a description of the analytic assembly map for the group $G$ at the level of connective spectra.
\end{theorem}

\begin{proof}
In the proof of theorem \ref{clasKt}, we have already commented that the functor $X\mapsto X_+ \wedge {\mathbb K}({\mathbb C})$ is excisive.  We have also seen that the $K$-theory assembly map is natural.  It follows that the composition
$$X_+ \wedge {\mathbb K}({\mathbb C}) \rightarrow {\mathbb K} {\mathbb C}\pi (X) \rightarrow {\mathbb K} C^\ast_\mathrm{max} \pi (X)$$
is natural.  This map is the identity map, and so is certainly a homotopy-equivalence, when the space $X$ is a single point.  The desired result now follows from theorem \ref{mitch3}.
\end{proof}

A similar result holds when working with the real numbers instead of the complex numbers; the proof is identical.

\section{The Equivariant Case}

We can use equivariant analogues of the techniques applied in the previous section to look at factorisations of the Baum-Connes assembly map introduced in \cite{BCH}.  Our factorisation is constructed with a number of techniques first introduced in \cite{DL}.

\begin{definition}
Let $G$ be a discrete group.  Then we define the {\em orbit category} of $G$ to be the category in which the objects are $G$-spaces of the form $G/H$, where $H$ is a subgroup of $G$, and morphisms are equivariant maps.
\end{definition}

Given a category $\mathcal C$, we call a contravariant functor from $\mathcal C$ to the category of pointed topological spaces a {\em right pointed $\mathcal C$-space}, and a covariant functor from $\mathcal C$ to the category of symmetric spectra a {\em left $\mathcal C$-spectrum}.

For example, given a pointed $G$-space $X$, there is an associated right pointed $\Or (G)$-space $\map_G (-,X)$ defined by mapping the object $G/H$ to the space of equivariant basepoint-preserving maps $\map_G (G/H , X)$.

\begin{definition}
Let $X$ be a right $\mathcal C$-space, and let $F$ be a left $\mathcal C$-spectrum.  Let $X_+$ be the right $\mathcal C$-spectrum with spaces $X_+(A) =X(A)_+$ for each object $A\in \Ob ({\mathcal C})$.  Then we define the {\em product} $X_+\wedge_{\mathcal C} F$ to be the spectrum with spaces
$$(X_+(A)\wedge_{\mathcal C}F)_n = \amalg_{A\in \Ob ({\mathcal C})} X_+(A) \wedge F(A)_n / \sim$$
where $\sim$ is the equivalence relation
$$x\wedge \phi y = x\phi \wedge y$$
for all elements $x\in X_+(B)$, $y\in F(A)_n$, and $\phi \in \Hom (A,B)_{\mathcal C}$,
\end{definition}

\begin{definition}
Let $F$ be a covariant functor from the category of right $G$-spaces to the stable category of symmetric spectra.  Then we call the functor $F$ {\em $G$-homotopy-invariant} if it preserves $G$-homotopy-equivalences.
\end{definition}

The following result comes from \cite{DL}.

\begin{theorem} \label{DLA}
Let $F$ be a functor from the category of proper $G$-$CW$-complexes to the category of spectra.  Then we can define a $G$-homotopy-invariant and excisive functor $F^\%$ by writing
$$F^\% (X) = \map_G (-,X_+)\wedge_{\Or (G)} F$$

Further, there is a canonical functor $\alpha \colon F^\% \rightarrow F$ which is an equivalence whenever the space $X$ takes the form $G/H$, where $H$ is a finite subgroup of $G$.  The $G$-homotopy-invariant functor $F^\%$ with this property and functor $\alpha$ are unique up to stable equivalence.
\noproof
\end{theorem}

The functor $\alpha$ in the above theorem is called the {\em equivariant assembly map} associated to the functor $F$.

\begin{definition}
Let $\mathcal C$ be a category.  Then a {\em topological $\mathcal C$-ring} is a functor from the category $\mathcal C$ to the category of topological rings.
\end{definition}

Given a topological $\mathcal C$-ring $R$, we can associate a left $\mathcal C$-spectrum ${\mathbb K}(R)$ by writing ${\mathbb K}(R)(A) = {\mathbb K}(R(A))$ for each object $A\in \Ob ({\mathcal C} )$.

\begin{definition}
Let $G$ be a discrete group, and let $X$ be a topological space equipped with a right $G$-action.  Then we define the {\em transport groupoid}, $\overline{X}$, to be the discrete groupoid in which the objects are the points of the space $X$, and the morphism sets are defined by writing
$$\Hom (x,y)_{\overline{X}} := \{ g\in G \ |\ xg=y \}$$
\end{definition}

\begin{definition}
Let $\pi$ be a discrete groupoid, and let $R$ be a topological $\pi$-ring.  Then we define $R\pi$ to be the category with objects $\Ob (R\pi ) = \Ob (\pi )$ and morphism sets
$$\Hom (a,b)_{R\pi} := \{ x_1 g_1 + \cdots + x_n g_n \ |\ x_i \in R_b, g_i \in \Hom (a,b)_\pi \}$$

Composition of morphisms is defined by the formula
$$\left( \sum_{i=1}^m x_i g_i \right) \left( \sum_{j=1}^n y_j h_j \right) = \sum_{i,j=1}^{m,n} x_i g_i(y_j ) g_i h_j$$

The topology on the space $\Hom (a,b)_{R\pi}$ is defined by viewing it as a direct limit of a family of subspaces homeomorphic to the space $R^n$ for some $n$.
\end{definition}

Let $R$ be a topological $G$-ring.  Then there is a canonical faithful functor $i\colon \overline{X} \rightarrow G$.  It follows that $R$ can also be considered a topological $\overline{X}$-ring.  The following result is easy to check.

\begin{proposition}
The functor defined by sending a proper $G$-$CW$-complex $X$ to the spectrum ${\mathbb K}(R\overline{X})$ is $G$-homotopy-invariant.
\noproof
\end{proposition}

Let ${\mathbb K}^G_\mathrm{hom}(-;R)$ be the excisive and $G$-homotopy-invariant functor defined by writing
$${\mathbb K}^G_\mathrm{hom} (X;R) = \map_G (-,X_+)\wedge_{\Or (G)} (Y\mapsto {\mathbb K}(R\overline{Y}))$$

Then by theorem \ref{DLA} there is an equivariant assembly map
$$\alpha \colon {\mathbb K}_\mathrm{hom}^G(X;R) \rightarrow {\mathbb K}(R\overline{X})$$

When $R = {\mathbb Z}$, according to \cite{DL}, the above assembly map is the map appearing in the Farrell-Jones conjecture of \cite{FJ}.

To fit the Baum-Connes assembly map into our present framework, we recall the main result of \cite{Mitch6}, which provides all we need to know about the Baum-Connes assembly map for the present purposes.

\begin{theorem} \label{mitch6}
Let $A$ be a $G$-$C^\ast$-algebra.  Let $F^\%$ be a $G$-homotopy-invariant and excisive functor from the category of right $G$-spaces to the stable category of symmetric spectra.  Let $\alpha \colon F^\% (X)\rightarrow {\mathbb K}(A\rtimes_r \overline{X})$ be a natural map of spectra that is a stable equivalence when the space $X$ has the form $G/H$, where $H$ is a finite subgroup of the group $G$.

Let $i\colon \overline{X}\rightarrow G$ be the canonical faithful functor.  Then up to stable equivalence the composite map $\alpha i_\star \colon F^\% (X)\rightarrow {\mathbb K}(A\rtimes_r G)$ is the Baum-Connes assembly map.
\noproof
\end{theorem}

Here, for a discrete groupoid $\pi$, a $\pi$-$C^\ast$-algebra is defined similarly to a topological $\pi$-ring, but has some extra structure.  The topological ringoid $A\rtimes_r \pi$ is a certain Banach category obtained by completing the topological ringoid $A\pi$; see \cite{Mitch6} for further details.  There is a canonical inclusion $A\pi \hookrightarrow A\rtimes_r\pi$.

\begin{theorem}
Let $G$ be a discrete group, and let $A$ be a $G$-$C^\ast$-algebra.  Then the composition
$${\mathbb K}_\mathrm{hom}^G(X;A) \rightarrow {\mathbb K} (A\overline{X}) \rightarrow {\mathbb K} (AG)\rightarrow {\mathbb K}(A\rtimes_r G)$$
is a description of the Baum-Connes assembly map for the group $G$ at the level of connective spectra.
\end{theorem}

\begin{proof}
We have a commutative diagram
$$\begin{array}{ccc}
{\mathbb K} (A\overline{X}) & \rightarrow & {\mathbb K} (AG) \\
\downarrow & & \downarrow \\
{\mathbb K}(A\rtimes_r \overline{X}) & \rightarrow & {\mathbb K}(A\rtimes_r G) \\
\end{array}$$
so our map can also be viewed as the composition
$${\mathbb K}_\mathrm{hom}^G(X;A) \rightarrow {\mathbb K} (A\overline{X}) \rightarrow {\mathbb K} (A\rtimes_r \overline{X})\rightarrow {\mathbb K}(A\rtimes_r G)$$

The above composition is natural.  When $X=G/H$, with $H$ a finite group, the categories $A\overline{X}$ and $A\rtimes_r \overline{X}$ are equal.  Thus the map
${\mathbb K}_\mathrm{hom}^G(X;A) \rightarrow {\mathbb K} (A\rtimes_r \overline{X})$ is a stable equivalence in this case, and the result follows from theorem \ref{mitch6}.
\end{proof}

\bibliographystyle{plain}

\begin{thebibliography}{10}

\bibitem{BCH}
P.~Baum, A.~Connes, and N.~Higson.
\newblock {Classifying spaces for proper actions and $K$-theory of group
  $C^\ast$-algebras}.
\newblock In S.~Doran, editor, {\em $C^\ast$-algebras: 1943--1993}, volume 167
  of {\em Contemporary Mathematics}, pages 241--291. American Mathematical
  Society, 1994.

\bibitem{DL}
J.~Davis and W.~L{\"u}ck.
\newblock {Spaces over a category and assembly maps in isomorphism conjectures
  in $K$- and $L$-theory}.
\newblock {\em $K$-theory}, 15:241--291, 1998.

\bibitem{FJ}
F.T. Farrell and L.~Jones.
\newblock {Isomorphism conjectures in algebraic $K$-theory}.
\newblock {\em Journal of the American Mathematical Society}, 6:377--392, 1993.

\bibitem{HSS}
M.~Hovey, B.~Shipley, and J.~Smith.
\newblock Symmetric spectra.
\newblock {\em Journal of the American Mathematical Society}, 13:149--208,
  2000.

\bibitem{Jo2}
M.~Joachim.
\newblock {$K$-homology of $C^\ast$-categories and symmetric spectra
  representing $K$-homology}.
\newblock {\em Mathematische Annalen}, 327:641--670, 2003.

\bibitem{Lod}
J.L. Loday.
\newblock {$K$-theorie alg{\`e}brique et repr{\'e}sentations des groupes}.
\newblock {\em Annales Scientifiques de l'{\'E}cole Normal Sup{\'e}rieure (4)},
  9:309--377, 1976.

\bibitem{May4}
J.P. May.
\newblock {The dual Whitehead theorems}.
\newblock In {\em Topological topics}, volume~86 of {\em London Mathematical
  Society Lecture Note Series}, pages 46--54. Cambridge University Press, 1983.

\bibitem{McC}
R.~McCarthy.
\newblock {On fundamental theorems of algebraic $K$-theory}.
\newblock {\em Topology}, 32:325--328, 1993.

\bibitem{McS}
D.~McDuff and G.~Segal.
\newblock Homology fibrations and the group completion theorem.
\newblock {\em Inventiones Mathematicae}, 31:279--284, 1975/1976.

\bibitem{Mil}
J.~Milnor.
\newblock {\em Morse Theory}, volume~51 of {\em Annals iof Mathematics
  Studies}.
\newblock Princeton University Press, 1973.

\bibitem{Mi}
B.~Mitchell.
\newblock {\em Separable algebroids}, volume 333 of {\em Memoirs of the
  American Mathematical Society}.
\newblock American Mathematical Society, 1985.

\bibitem{Mitch2.5}
P.D. Mitchener.
\newblock {Symmetric $K$-theory spectra of $C^\ast$-categories}.
\newblock {\em $K$-theory}, 24:157--201, 2001.

\bibitem{Mitch2}
P.D. Mitchener.
\newblock {$C^\ast$-categories}.
\newblock {\em Proceedings of the London Mathematical Society}, 84:374--404,
  2002.

\bibitem{Mitch3}
P.D. Mitchener.
\newblock {$KK$-theory of $C^\ast$-categories and the analytic assembly map}.
\newblock {\em $K$-theory}, 26:307--344, 2002.

\bibitem{Mitch6}
P.D. Mitchener.
\newblock {$C^\ast$-categories, groupoid actions, equivariant $KK$-theory, and
  the Baum-Connes conjecture}.
\newblock {\em Journal of Functional Analysis}, 214:1--39, 2004.

\bibitem{Q1}
D.~Quillen.
\newblock {Higher algebraic $K$-theory I}.
\newblock In {\em {Algebraic $K$-theory , I: Higher $K$-theories}}, volume 341
  of {\em Lecture Notes in Mathematics}, pages 77--139. Springer, Berlin, 1973.

\bibitem{RW}
J.~Rosenberg and S.~Weinberger.
\newblock {An equivariant Novikov conjecture}.
\newblock {\em $K$-theory}, 4:29--53, 1990.

\bibitem{Seg2}
G.~Segal.
\newblock Categories and cohomology theories.
\newblock {\em Topology}, 13:293--312, 1974.

\bibitem{Sta}
R.E. Staffeldt.
\newblock {On fundamental theorems of algebraic $K$-theory}.
\newblock {\em $K$-theory}, 2:511--532, 1989.

\bibitem{Wa1}
F.~Waldhausen.
\newblock {Algebraic $K$-theory of spaces}.
\newblock In {\em Algebraic and geometric topology (New Brunswick, N.J.,
  1983)}, volume 1126 of {\em Lecture Notes in Mathematics}, pages 318--419.
  Springer, Berlin, 1985.

\bibitem{Wa3}
F.~Waldhausen.
\newblock {Algebraic $K$-theory of spaces, a manifold approach}.
\newblock In {\em Current trends in algebraic topology, Part 1 (London,
  Ontario, 1981)}, volume~2 of {\em CMS Conference Proceedings}, pages
  141--184. American Mathematical Society, 1985.

\bibitem{WW}
M.~Weiss and B.~Williams.
\newblock Assembly.
\newblock In {\em Novikov Conjectures, Index Theorems, and Rigidity, Volume 2
  (Oberwolfach 1993)}, volume 227 of {\em London Mathematical Society Lecture
  Note Series}, pages 332--352. Cambridge University Press, 1995.

\end{thebibliography}

\end{document}